\documentclass[12 pt]{amsart}

\usepackage{graphicx}
\usepackage{amssymb,latexsym}
\usepackage{amsmath,amsfonts,amstext}

\newtheorem{theorem}{Theorem}
\newtheorem{lemma}{Lemma}
\newtheorem{proposition}{Proposition}
\newtheorem{definition}{Definition}
\newtheorem{corollary}{Corollary}

\begin{document}
\title[ENDLINE BILINEAR CONE RESTRICTION ESTIMATE]
      {An endline bilinear cone restriction estimate for mixed norms}
      
\author{Faruk Temur}
\address{Department of Mathematics\\
        University of Illinois at Urbana-Champaign
        Urbana, IL 61820}
\email{temur1@illinois.edu}
\thanks{The author was supported by NSF grant DMS-0900865. The author thanks to his advisor M.B. Erdo\~{g}an for the financial support}
\date{November 5, 2010}    

\begin{abstract}
We prove an $L^2 \times L^2 \rightarrow L_t^qL_x^p $ bilinear Fourier extension estimate for the cone when $p,q$ are on the critical line $1/q=(\frac{n+1}{2})(1-1/p)$. This extends previous results by Wolff, Tao and Lee-Vargas. 
\end{abstract}

\maketitle

\section{Introduction}\label{intro}

Let $n\geq1$ be an integer and let $f \in L^p(\textbf{R}^{n+1})$. If $p=1$, then $\widehat{f}$ is a continuous function, and hence can be restricted to any hypersurface; whereas if $p=2$,  $\widehat{f}$ can be an arbitrary square integrable function meaning that it cannot be restricted to any measure zero set.  Restriction problem concerns what happens if $1<p<2$.
It is easy to concoct examples showing that we cannot meaningfully restrict $\widehat{f}$ to a hyperplane if $p>1$. So it was a surprising discovery when Stein in 1967 observed that for curved hypersurfaces the situation is different, and one can, for certain values of $p$ depending on the surface chosen, restrict $\widehat{f}$. Restriction problem, essentially, is to determine the range of $p$.

The rigorous formulation of this problem is as follows. Let S be a smooth compact hypersurface with boundary in $\textbf{R}^{n+1}$. We say that the linear restriction estimate $R_S(p \rightarrow q)$ holds if 
\begin{equation}
\|\widehat{f} \|_{L^q(S,d\sigma)} \leq C_{p,q,S}\|f\|_{L^{p}(\textbf{R}^{n+1})}
\end{equation}for all Schwartz functions $f$ on $\textbf{R}^{n+1}$. Equivalently one can use the following formulation. We say that that linear adjoint restriction estimate $R^*_S(p \rightarrow q)$ holds if 
\begin{equation}\label{eqi2}
\|\widehat{fd\sigma} \|_{L^q(\textbf{R}^{n+1})} \leq C_{p,q,S}\|f\|_{L^{p}(S,d\sigma)}
\end{equation}
for all $C^{\infty}$ functions $f$ on $S$. 

This problem was posed by Stein in \cite{S1}. It is well understood for $n=1$ but wide open  for $n\geq2$. It is known to be connected to other central problems in harmonic analysis such as the Kakeya conjecture and the Bochner-Riesz conjecture; for the exact nature of these connections see e.g. \cite{W2}, \cite{La}, \cite{T3}, \cite{JB2}.

Consider the following variant of this problem: Let $S_1, S_2$ be two smooth compact hypersurfaces in $\textbf{R}^{n+1}$ with Lebesgue measure $d\sigma_1$ and $d\sigma_2$ respectively. We say that bilinear adjoint restriction estimate $R^*_{S_1,S_2}(2\times2 \rightarrow q)$ holds if one has 
\begin{equation}\label{be}
\|\widehat{f_1d\sigma_1}\widehat{f_2d\sigma_2}\|_{L^q(\textbf{R}^{n+1})}\leq C_{q,S_1,S_2}\|f_1\|_{L^2(S_1,d\sigma_1)}\|f_2\|_{L^2(S_2,d\sigma_2)}
\end{equation}
for all smooth functions $f_1, f_2$ supported respectively on $S_1,S_2$ . Historically the first incentive to study this problem  was to attack (\ref{eqi2}) in the special case $q=4$ by squaring both sides  and studying the resulting bilinear estimate; see e.g. \cite{F}, \cite{Sj}. Later this idea was extended to other values of $q$. In \cite{JB1}, \cite{TVV} it was observed that  if $S_1, S_2$  satisfy certain transversality conditions,  further estimates that are not available for arbitrary $S_1,S_2$ are available. What is more, these estimates then can be used to obtain new linear restriction estimates; see \cite{TVV}, \cite{TV1}, \cite{W1}. These advantages motivate study of this type of restriction estimates, which at a first look seems more complicated and  less hopeful. 

Let $S_1, S_2$ be compact, transverse subsets of the light cone
\[\{(x,t)\in \textbf{R}^{n+1}: |x|=|t| \}\]
or compact, transverse subsets of the paraboloid
\[\{(x,t)\in \textbf{R}^{n+1}: t=-\frac{1}{2}|x|^2\}.\]
The study of bilinear extension estimates for such $S_1, S_2$ dates back to  Carleson-Sj\"{o}lin Theorem, which states that $R^*_{S_1,S_2}(2\times2 \rightarrow q)$ holds for $q=2$ when $n=1$; see \cite{CS}. This theorem is known to be optimal, that is, for $n=1$ going below  $q=2$ is not possible.  For the cone case,  Bourgain proved that going below the exponent $2$ is possible when $n=2$, and that $q\geq 2-13/2048$ is enough; see \cite{JB1}. Then in 1997 Klainerman and Machedon observed that for $n\geq 2$ the condition $q \geq \frac{n+3}{n+1}$ is necessary, and conjectured that this condition suffices. 

For the cone case further progress came from Tao and Vargas, who proved that when $n=2$, it suffices to have $q \geq 2-8/121$  in \cite{TV1}. Then Wolff made a great breakthrough and settled the conjecture of Klainerman and Machedon except for the endpoint in \cite{W1}, with the endpoint being attained shortly afterwards by Tao in \cite{T1}.  

In the paraboloid case, first  progress came when Tao, Vargas and Vega in \cite{TVV} proved that for the special case $n=2$,  $q\geq 2-5/69$ suffices. Tao and Vargas furthered this to $q \geq 2-2/17$ in \cite{TV1}.  Finally  in \cite{T2} Tao proved the conjecture except for the endpoint.  Endpoint is known only for the cylindrically symmetric case; see \cite{SS}.

Resolution of Klainerman-Machedon conjecture is important also  for its applications to various problems. One important application is to PDE. The cone and the paraboloid are related to solutions of the wave equation and the Schr\"{o}dinger equation respectively.  One can thus reformulate these results in terms of solutions of the wave equation and the Schr\"{o}dinger equation, and then apply these to obtain null form estimates, which, in turn, are important for the study of nonlinear PDE; see \cite{FK},  \cite{T1}, \cite{DT}, \cite{LV1}, \cite{LVR}. Actually we will also use this PDE formulation in this paper. It is this connection that motivated study of these conjectures for mixed norms, and progress has been made in \cite{LV1} in this direction. A second application is to the Bochner-Riesz conjecture: the progress in the paraboloid case was used to obtain the best known exponent for the Bochner-Riesz problem; see \cite{LBR}. Another application is  to the Falconer distance set problem. Via proving a weighted  bilinear restriction estimate  for the paraboloid, Erdo\~{g}an improved the known bound for the distance set problem; see \cite{BE1},\cite{BE2}. This idea also applied to obtain bounds for distance sets defined with respect to non-Euclidean distance functions; see \cite{La} and references therein.

There has also been some effort to extend the known  results for the cone and the paraboloid to more general curved surfaces; see  \cite{Lee}, \cite{V1}.
         
We now restrict our attention to the cone case which our result concerns. We reformulate the problem as mentioned above. Let a function $\phi :\textbf{R}^{n+1}\rightarrow H$ be a red wave if $H$ is a finite dimensional complex Hilbert space , and  if its spacetime Fourier transform $\widehat{\phi}$ is an $L^2$ measure on the set
\[ \Sigma^{R}:=\left\{ (\xi,|\xi|) \in \textbf{R}^{n+1}: \angle(\xi,e_1) \leq \pi/8 ,\   1 \leq |\xi| \leq 2 \right\}\]
where $e_1$ is a fixed basis vector. Similarly let a function  $\psi :\textbf{R}^{n+1}\rightarrow H'$ be a blue wave  if $H'$ is a finite dimensional complex Hilbert space , and if $\widehat{\psi}$ is an $L^2$ measure on the set
\[ \Sigma^{B}:=\left\{ (\xi,-|\xi|) \in \textbf{R}^{n+1}: \angle(\xi,e_1) \leq \pi/8 ,\  1 \leq |\xi| \leq 2 \right\}.\]
Let \textit{energy} for red and blue waves be defined as 
\begin{equation}\label{eq1.1}
	E(\phi):=\left\| \phi(t) \right\|^2_2,\  E(\psi):=\left\| \psi(t) \right\|^2_2 
\end{equation}
where $\phi(t)$, $\psi(t)$ are given by $\phi(t)(x):=\phi(x,t)$, $\psi(t)(x):=\psi(x,t)$. This definition is independent of time $t$. Then by the results of Wolff and Tao 
\begin{equation}\label{e1.2}
 \left\|\phi \psi \right\|_p \lesssim  E(\phi)^{1/2}E(\psi)^{1/2} 
\end{equation}
holds for $p \geq \frac{n+3}{n+1}$. Here and in what follows implicit constants do not depend on $H, H'$, and $\phi\psi:\textbf{R}^{n+1} \rightarrow H\otimes H'$ where $\otimes$ denotes the tensor product. The following is our main theorem.
\begin{theorem}\label{t1.2}

Let $\phi,\psi$ be respectively red and blue waves. Let $1 < p,q \leq 2$ be such that $1/q \leq (\frac{n+1}{2})(1-1/p)$ and $1/q <\min(1,\frac{n+1}{4})$. Then we have
\begin{equation}\label{eqi.3}
 \left\|\phi \psi \right\|_{L_t^qL_x^p} \lesssim E(\phi)^{1/2}E(\psi)^{1/2} .
\end{equation}

\end{theorem}
 
 Lee and Vargas proved this theorem with  $1/q < (\frac{n+1}{2})(1-1/p)$ in \cite{LV1}, so we extend that result to the endline. 
 
 We now describe examples showing that the conditions $1/q \leq (\frac{n+1}{2})(1-1/p)$, $1/q <\min(1,\frac{n+1}{4})$ are necessary. These examples are similar to the ones given in \cite{FK}.
 
 \vspace{3mm} 
 
\begin{flushleft}\textbf{Example 1.} Let's denote an element $\xi \in \textbf{R}^{n}$  as $\xi =(\xi_1, \xi')$. Then let \end{flushleft} 
\[S:= \{ \xi \in \textbf{R}^{n}: |\xi_1-3/2|<\epsilon^2, |\xi'|<\epsilon \}.\]
Then one has $|S|\approx \epsilon^{n+1}$. Define two functions $R,B$ on respectively on $\Sigma^{R}$ and $\Sigma^{B}$ as follows
\[R(\xi,|\xi|):=\chi_S(\xi), \ \ \ \  B(\xi,-|\xi|):=\chi_S(\xi).\]
Thus, these are the characteristic functions of projections of the set $S$ to $\Sigma^{R}$ and $\Sigma^{B}$ respectively. Let $d\sigma$ denote the surface measure of the cone. Define red and blue waves $\phi, \psi$ by
\[\widehat{\phi}:=Rd\sigma, \ \ \ \ \widehat{\psi}:=Bd\sigma.\]
So we have 
\[\|\phi\|_2=\|\widehat{\phi}\|_2 \approx \epsilon^{\frac{n+1}{2}}\]
and similarly
\[\|\psi\|_2=\|\widehat{\psi}\|_2 \approx \epsilon^{\frac{n+1}{2}}.\]
On the other hand by the uncertainty principle both $|\phi|,|\psi|$ are comparable to $\epsilon^{n+1}$ on a rectangular box that has a spatial area $\approx \epsilon^{-(n+1)}$ for $|t| \lesssim  \epsilon^{-2} $.
Then one obtains
\[ \|\phi\psi\|_{L_t^qL_x^p} \gtrsim \epsilon^{2n+2}\epsilon^{-2/q}\epsilon^{-(n+1)/p}.\]
Thus we need
\[\epsilon^{2n+2}\epsilon^{-2/q}\epsilon^{-(n+1)/p} \lesssim \epsilon^{n+1}\]
which implies
\[\frac{1}{q}\leq (\frac{n+1}{2})(1-\frac{1}{p}).\]

\vspace{3mm}
\begin{flushleft}\textbf{Example 2.} Let \end{flushleft}
\[
 S_1:=\{\xi \in \textbf{R}^{n}: |\xi_1-3/2|<1/4, |\xi'|<\epsilon \}. \]
Let $S_2$ be the set formed by intersection of a space-time slab of thickness $\epsilon^2$ whose normal is parallel to that of $S_1$ with $\Sigma^B$.  
Then  we have $|S_1|\approx \epsilon^{n-1}$ and $|S_2|\approx \epsilon^2$. Define two functions $R,B$ on respectively on $\Sigma^{R}$ and $\Sigma^{B}$ as follows
\[R(\xi,|\xi|):=\chi_{S_1}(\xi), \ \ \ \  B:=\chi_{S_2}.\]
Thus, $R$ is the characteristic function of projection of the set $S_1$ to $\Sigma^{R}$, and $B$ is the characteristic function of  $S_2$. Let $d\sigma$ denote the surface measure of the cone.
Define red and blue waves $\phi, \psi$ by
\[\widehat{\phi}:=Rd\sigma, \ \ \ \ \widehat{\psi}:=Bd\sigma.\]
So we have 
\[\|\phi\|_2=\|\widehat{\phi}\|_2 \approx \epsilon^{\frac{n-1}{2}}\]
and similarly
\[\|\psi\|_2=\|\widehat{\psi}\|_2 \approx \epsilon.\]
On the other hand by the uncertainty principle  $|\phi|$ is comparable to $\epsilon^{n-1}$, and $|\psi|$ is comparable to $\epsilon^2$ on a rectangular box that has spatial area comparable to $1$ for all $t \lesssim \epsilon^{-2}$.
Thus we obtain
\[\|\phi\psi\|_{L_t^qL_x^p} \gtrsim \epsilon^{n+1}\epsilon^{-2/q}.\]
Hence we need
\[\epsilon^{n+1}\epsilon^{-2/q} \leq \epsilon^{(n+1)/2}.\]
This implies
\[\frac{1}{q} \leq \frac{n+1}{4}.\]

\vspace{2mm}

\begin{flushleft}
\textbf{\textit{List of Notation}}\hfill

$D=D(x_D,t_D;r_D):=\{ (x,t_D): |x-x_D|\leq r_D \}.$ \hfill

$D^{ext}=D^{ext}(x_D,t_D;r_D):= \{(x,t_D): |x-x_D|>r_D)\}.$\hfill

$Q(x_Q,t_Q;r_Q):$  ${n+1}$ dimensional cube in $\textbf{ R}^{n+1}$ centered at $(x_Q,t_Q)$ with side-legth $r_D$ and  sides parallel to axes. Life-span for such a cube is defined to be the interval $[t_Q-\frac{1}{2}r_Q, t_Q+\frac{1}{2}r_Q]$.

$cQ:=Q(x_Q,t_Q;cr_Q)$

$Q^{ann}(x_Q,t_Q;r_1,r_2):=Q(x_Q,t_Q;r_2)\setminus Q(x_Q,t_Q;r_1)$

$\underline{\Sigma}:= \{ \xi \in \textbf{R}^{n} : 1/2\leq |\xi| \leq 4, \angle(\xi,e_1)\leq \pi/4 \}.$

$C^{R}(x_0,t_0):=\{(x_0+r\omega,t_0-r) \in \textbf{R}^{n+1}:r\in R, \omega \in S^{n-1} \cap \underline{\Sigma} \}.$

$C^{B}(x_0,t_0):=\{(x_0+r\omega,t_0+r) \in \textbf{R}^{n+1}:r\in R, \omega \in S^{n-1} \cap \underline{\Sigma} \}.$

$C^{R}(x_0,t_0;r),C^{B}(x_0,t_0;r): r$ neighborhoods of $C^{R}(x_0,t_0), C^{R}(x_0,t_0)$ respectively.

$C^P(x_0,t_0;r):=C^{R}(x_0,t_0;r)\cup C^{B}(x_0,t_0;r)$

\textbf{T}: The time reversal operator given by $\textbf{T}\phi(x,t)=\phi(x,-t).$

\end{flushleft}


\section{Preliminaries}\label{s2}

Our proof will mainly follow Tao's proof in \cite{T1}. Here is a sketch of the proof in which we will gloss over tecnical details, and try to convey the main ideas of this  complicated proof. First we will localize the estimate (\ref{eqi.3}) to cubes of side-length $R$.  Then by monotone convergence theorem finding a bound independent of $R$ suffices. We observe that once localized, using the definition of energy and the H\"{o}lder inequality we can obtain a trivial $L^1$ estimate. So if we can prove a favorable $L^2$ estimate, using the H\"{o}lder inequality and interpolation we can control the localized forms of (\ref{eqi.3}). We are not able to prove such an estimate directly for our waves, but we can still apply this strategy partially as follows. We localize our waves to sub-cubes using the standard tool of wave packet decomposition. Then for waves localized to sub-cubes we can obtain favorable $L^2$ estimates on other sub-cubes, which we do by using the wave packet decomposition,  and applying our strategy above yields a constant term. But still we need to estimate localized waves on cubes to which they localized. In this case we do not estimate them at all, and just use the fact that we are estimating them at a lower scale. Thus we are able to bound an estimate at a scale with an estimate at a lower scale plus  a constant.  This is the induction on scales method of Wolff,  using this technique Wolff settled Klainerman-Machedon conjecture for the cone except for the endpoint; see \cite{W1}. Yet    due to endpoint/endline nature of the problem we  have a constant instead of a  negative power of $R$ in the error term, which is the case for the non-endpoint/non-endline problem, so we are not able to run an induction on scales.   Arguments used up to this point are enough to prove non-endpoint or non-endline results, and are used in \cite{W1}, \cite{LV1} to obtain these results.

At this stage one needs the observation that it is the concentration of energy of both $\phi$ and $\psi$  in a disk of small radius compared to side-length of our cubes that troubles us. In the absence of this problem, one can improve the $L^2$ estimate slightly to obtain the endpoint/endline results. But it is also certain that one cannot escape concentration, for concentration as defined here depends on the side-length of the cube to which we localize:   as the scale gets larger previously non-concentrated waves become concentrated. So concentration  must also be dealt with. To do this one needs a second observation: if both $\phi$ and $\psi$  concentrate on a disk then $\phi\psi$ concentrate on the double light cone generated by that disk. This phenomenon is called Huygens' principle. Restricting to this set one can get a better $L^1$ estimate due to transversality of Fourier supports of $\phi$ and $\psi$. This allows one to obtain a better error term when controlling the estimate at scale  $R$ with estimate at a lower scale, and hence do induction on scales without the problem one encounters in the process described   above. This means that if we can pass back and forth between estimates for cubes and estimates for cone neighborhoods, we can exploit this second observation to  deal with the concentration. Having dealt with both concentrated and non-concentrated cases, we combine them to obtain a uniform bound on localized estimates for the cubes. The two observations above are due to Tao and using these he proved the endpoint case of Klainerman-Machedon conjecture for the cone; see \cite{T1}.

Now we perform some reductions.  Firstly, it is clear that it suffices to prove Theorem 1.1 for waves satisfying the energy normalization
\[E(\phi), E(\psi)=1.\]
We will exploit this in some of our propositions. Secondly, it suffices to prove our theorem only for the endline, since for the non-endline cases it is already known  by \cite{LV1}. Finally, observe that
\begin{equation}\label{eq2.1}
\begin{aligned}
\|\phi \psi\|_{L^{\infty}_tL^{1}_x} =\sup_t \int | \phi(t) \psi(t)|dx &\leq \sup_t \|\phi(t)\|_{L^2}\|\psi(t)\|_{L^2}\\ &\leq  E(\phi)^{1/2}E(\psi)^{1/2} .
 \end{aligned}
\end{equation}
So if (\ref{eqi.3}) is correct for $1< p,q < \infty$ such that $1/q <\min (1,(n+1)/4), \ 1/q=\frac{n+1}{2}(1-\frac{1}{p})$, then interpolating with (\ref{eq2.1}) it is correct for each point on the line  between $(1,\infty)$ and $(p,q)$. So it is enough to prove that (\ref{eqi.3})  holds for $(p,q)$ with $q$ arbitrarily close to $\min (1,(n+1)/4)$. Hence for any $n \geq 2$ fix $0< \epsilon < \frac{1}{10n}$ and let $1 \leq p,q \leq \infty$ be such that  $1/q=\min (1,(n+1)/4)-\epsilon$, $1/q=\frac{n+1}{2}(1-\frac{1}{p})$. Requirement on $\epsilon $ ensures $p>q$, so we will be able to use Lemma \ref{l2.1} below.
 
 Now we introduce the constants that will be used throughout the proof. Let $N$ denote the large integer $N=2^{n^{10}}$, thus N depends only on n. Let $C_0=2^{\left\lfloor N/{\epsilon}\right\rfloor^{10}}$. So $C_0$ is much larger than $N,1/\epsilon$. We will also use the following much larger constant: $C_1=2^{{C_0}^{10}}$. Throughout the proof $C$ will denote various large numbers that vary from line to line and that may depend on $N, \ \epsilon$ but not on $C_0$ and the dimension of $H,H'$. That $C$ may depend on $\epsilon$ is not a problem since implicit constant of (\ref{eqi.3}) also depends on $(p,q)$. So we have $C<N<C_0<C_1$ and each of these quantities dominates any reasonable quantity arising from quantities smaller than it. We shall use $A \lesssim B$ to mean $A \leq CB$, and $A \approx B$ for $A \lesssim B$ and $B \lesssim A$.
 
 Let's examine red and blue waves more closely. Clearly blue waves are time reversals of red waves. Both red and blue waves solve free wave equation, but propagate along different sets of characteristics.   Due to compact Fourier supports waves are smooth and bounded. The following machinery will help us understand propagation of waves. Since the wave equation is a constant coefficient linear PDE, using the Fourier transform for space variables we can obtain an evolution operator defined by the Fourier transform that given the initial value of a wave, allows one to calculate it at another time. For our red waves this operator takes the following form. 
Let $a(\xi)$ be a fixed bump function supported on $\underline{\Sigma}$ which is equal to 1 on the spatial projection of $\Sigma^{R}$ and $\Sigma^{B}$. Then the evolution operator is $U(t)$  defined by
\[ \widehat{U(t)f}(\xi):= a(\xi)e^{2\pi it|\xi|}\widehat{f}(\xi). \] 
As this evolution operator is defined by multiplication in the frequency space, it  will translate to a convolution with a kernel $K_t$ when we inverse Fourier transform both sides. This convolution kernel is given  by
\[K_t(x)=\int a(\xi)e^{2\pi i(x\cdot \xi +t|\xi| )}d\xi.\] 
Thus we have the following equality
\begin{equation}\label{eq.fi}\phi(t)=U(t)\phi(0)=\phi(0)\ast K_t 
\end{equation}
for  all red waves , and all times  $t$. We want to have a decay estimate for our convolution kernel. Fixing a direction in space-time, and using non-stationary phase, we  can obtain decay estimates that depend on the distance of the point to  $C^{R}(0,0)$. More precisely 
\begin{equation}\label{eq.k}|K_t(x) |  \lesssim (1+dist((x,t),C^{R}(0,0)))^{-N^{10}.}
\end{equation}
 Using Young's inequality, and the estimate above we have $\| \phi(0) \|_{\infty} \lesssim E(\phi)^{1/2}$. Then by time translation invariance and time reversal symmetry we have
\begin{equation}\label{infty}
\| \phi \|_{\infty} \lesssim E(\phi)^{1/2}, \  \|\psi\|_{\infty} \lesssim E(\psi)^{1/2}
\end{equation}
and hence
\begin{equation}\label{infty2}
\| \phi \psi \|_{\infty} \lesssim E(\phi)^{1/2}E(\psi)^{1/2}.
\end{equation}

We shall decompose our waves into smaller waves.  To  ensure that these are still waves of the same color, we need to define margin of a wave, and proceed for waves that obey a margin requirement, since our decompositions slightly enlarge Fourier support of a wave. So let $margin(\phi)$ denote the quantity 
\[ margin(\phi):=dist(supp(\hat{\phi}), \partial \Sigma^{red})\]  
We define the margin of a blue wave analogously.

  We are ready to localize to cubes of side-length $R$.
   
  \begin{definition}\label{d2.1}  
  For any $R\geq C_02^{C_1 /2}$, let A(R) be the best constant for which the inequality
\[\left\| \phi \psi\right\|_{L^q_t L^p_x (Q)} \leq A(R)E(\phi)^{1/2}E(\psi)^{1/2} \]
holds for all spacetime cubes Q of side-length R, red waves $\phi$ and  blue waves $\psi$  of margin
\begin{equation}\label{smr}
margin(\phi), margin (\psi) \geq 1/100- (1/R)^{1/N}.   
\end{equation}

\end{definition}
 
 It is clear that $A(R)$ is finite for each $R$, e.g using (\ref{infty}) we have the following crude bound 
\begin{equation}\label{pb}
\|\phi\psi\|_{L^q_tL^p_x(Q_R)} \lesssim R^C E(\phi)^{1/2}E(\psi)^{1/2}
\end{equation}
 which in particular shows
 \begin{equation}\label{pba}
 A(R) \lesssim R^C.
 \end{equation}  
 Moreover via a finite decomposition of space and frequency, and some  Lorentz transforms we see that 
 \begin{equation}\label{eq.r'}
 \|\phi\psi\|_{L^q_tL^p_x(Q_R)} \lesssim A(R')E(\phi)^{1/2}E(\psi)^{1/2}
 \end{equation}
 for any cube $Q_R$ of side-length $R$, any $R \approx R'$, and any red and blue waves $\phi, \psi.$
   Hence it is enough to prove that
 \begin{equation}\label{eq2.2} 
 A(R) \lesssim 2^{CC_1}
 \end{equation}
 uniformly for all $R\geq C_02^{C_1}$. As $R$ gets larger the margin requirement becomes more strict, thus we will also need the following variant of $A(R)$
 \[\overline{A}(R):= \sup_{C_02^{C_1/2}\leq r\leq R} A(r).\]
 
 Now let's see some easy estimates that we get when we localize to a cube. Let $Q_R$ be a cube of side-length R, and $\phi,\psi$ arbitrary red and blue waves. 
\begin{equation}\label{ce} \|\phi\|_{L^{2}(Q_R)} \lesssim R^{1/2}E(\phi)^{1/2}, \\ \|\psi\|_{L^2(Q_R)}\lesssim R^{1/2}E(\psi)^{1/2} . \end{equation}
To obtain these just integrate energy along the life-span of the cube $Q_R$. Using H\"{o}lder these two gives
\begin{equation}\label{eq3.12}
\|\phi\psi \|_{L^1(Q_R)} \lesssim R E(\phi)^{1/2}E(\psi)^{1/2}.
\end{equation}

 Rest of the paper is organized as follows: in section \ref{s3} we will give the definitions and estimates necessary for the rest of the paper. Some of the material employed for this purpose is proved in \cite{T1}, and works without any important change in our case. Then in section \ref{s4} we will localize waves to subcubes, and prove our  key proposition which is our main tool for the rest of the proof. In section \ref{s5} the concept of energy concentration will be made precise, machinery to deal with the concentrated and the non-concentrated cases developed, and Theorem \ref{t1.2}  proved.


\section{ Necessary definitions and estimates}\label{s3}

 We first give definitions that will make precise localization to sub-cubes. Let $Q$ be a cube of side-length $R$. Let $K_j(Q)$ denote collection of all sub-cubes we obtain when we partition $Q$ into cubes of sidelength $2^{-j}R$. Of course, there are $2^{(n+1)j}$ sub-cubes in this collection. We define a $red \ wave \ table \  \phi \ of \ depth \  j$ on $Q$ to be any red wave with the vector form
  \[\phi:=(\phi^{(q)})_{q\in K_j(Q), }\]
  where components $\phi$  are also red waves. Note that by the definition of energy 

\begin{equation}\label{eq3.1}
E(\phi)= \sum_{q \in K_j(Q)} E(\phi^{(q)}).
\end{equation}
  For a red wave table $\phi$ of depth $j$ on $Q$ we define the $j-quilt \ [\phi]_{j}$ of $\phi$ to be the function
  \[[\phi]_{j}:= \sum_{q\in K_{j}(Q)} |\phi^{(q)}| \chi_q.\]
Hence we have the pointwise estimates
\begin{equation}\label{eq3.2}
|\phi^{(q)}| \chi_q \leq [\phi]_{j}\leq |\phi| \chi_{Q}
\end{equation}
   for all $q \in K_{j}(Q)$. 
   
   We define $(c,k) \ interior \ I^{c,k}(Q)$ of $Q$ for $k$ a nonnegative integer and $0<c\ll 1$ by 
   \[I^{c,k}(Q):= \bigcup_{q\in K_k(Q)} (1-c)q. \]

Next we give some definitions concerning localization of energy to a disk. Let $\eta_0$ denote a fixed non-negative Schwarz function on $\textbf{R}^n$ with total mass 1 whose Fourier transform is supported on the unit disk. For any $r>0$ let $\eta_r(x):=r^{-n}\eta_0(x/r)$. 

Let $D=D(x_D,t_D;r)$ be any disk. Then define the operator $P_D$ as follows: for any red wave $\phi$ let
\[ P_D\phi(t_D):=(\chi_D \ast \eta_{r^{1-1/N}})\phi(t_D). \]
when $t=t_D$ and 
\[P_D\phi(t)=U(t-t_D)P_D\phi(t_D)\] 
at other times $t$. It is easy to see that $P_D\phi$ is a red wave. To extend this definition to blue waves we use time reversal operator $\textbf{T}$:
\[ P_D\textbf{T}\phi:=\textbf{T}P_D\phi, \]
where 
\[\textbf{T}\phi(x,t):=\phi(x,-t).\]
Next lemma  shows that $P_D$ localize a wave to the disk $D$, and $(1-P_D)$ localizes to $D^{ext}$.

 \begin{lemma}[Lemma 10.2 in \cite{T1}]\label{l3.1}
  Let  $r\geq C_0$, and $D=D(x_D,t_D;r)$ be a disk. Let $\phi$ be a red wave such that  $margin(\phi)\geq C_0r^{-1+1/N}.$ Then $P_D\phi$ is a red wave  which satisfies the following margin and energy estimates:
\begin{equation}\label{eq3.3}
margin(P_D\phi)\geq margin(\phi)-Cr^{-1+1/N}
\end{equation}
 \begin{equation}\label{eq3.4}
\| \widetilde{\chi}_D^{-N} P_D \phi \|_{L^2(D^{ext}_+)}\lesssim r^{-N^2}E(\phi)^{1/2}
\end{equation}
 \begin{equation}\label{eq3.5}
\|(1-P_D)\phi \|_{L^2(D_-)} \lesssim r^{-N}E(\phi)^{1/2} 
\end{equation}
\begin{equation}\label{eq3.6}
E(P_D\phi) \leq \|\phi \|^2_{L^2(D_+)}+Cr^{-N}E(\phi)
\end{equation}
\begin{equation}\label{eq3.7}
E((1-P_D)\phi) \leq \|\phi \|^2_{L^2(D_-^{ext})}+Cr^{-N}E(\phi) 
\end{equation}
\begin{equation}\label{eq3.8}
E(P_D\phi),E((1-P_D)\phi) \leq E(\phi)
\end{equation}
where $D_-,D_+$ are the disks $D_{\pm}:=D(x_D,t_D;r(1 \pm r^{-1/2N} ))$.
 
 \end{lemma}

\begin{proof}
Margin estimate is clear. Notice that 
\[0 \leq \chi_D \ast \eta_{r^{1-1/N}}(x) \leq 1 \]
for all $x \in \textbf{R}^n$, thus we have (\ref{eq3.8}). For $x \in D^{ext}_+$
\[\widetilde{\chi}^{-N}_D(x)(\chi_D \ast \eta_{r^{1-1/N}})(x) \lesssim r^{-N^2} \]
thus follows (\ref{eq3.4}),(\ref{eq3.6}). For $x \in D_-$
\[\chi_D \ast \eta_{r^{1-1/N}}(x) \geq 1-Cr^{-N}\]
and hence we have (\ref{eq3.5}),(\ref{eq3.7}).

\end{proof}
 Analogue of this for blue waves is of course legitimate by time reversal. After looking at localization properties of $P_D$ at time $t_D$ we now explore localization of it in space-time.
 
 \begin{lemma}[See Lemma 10.3 in \cite{T1}]\label{l3.2} Let D be a disk of radius $r\geq 2^{C_0}$, and  $\phi$ be a red wave with $margin(\phi) \geq C_0r^{-1+1/N}$. Let $1\leq q<p\leq2$ and $r \lesssim R$ . Then if $\psi$ is an arbitrary blue wave, we have the finite speed of propagation law                           
\begin{equation}\label{eq3.9}
\|((1-P_D)\phi)\psi\|_{L_t^qL_x^p(Q(x_D,t_D,C^{-1}r))} \lesssim r^{C-N}E(\phi)^{1/2}E(\psi)^{1/2}
 \end{equation}
 \textit{and the Huygens' principle}
\begin{equation}\label{eq3.10}
\| (P_D\phi)\psi \|_{L_t^q L^p_x(Q(x_0,t_0;R)\setminus C^{R}(x_D,t_D;Cr+R^{1/N}))} \lesssim R^{C-N}E(\phi)^{1/2}E(\psi)^{1/2}.  
 \end{equation}
\textit{where $x_0\in \textbf{R}^{n}$ is arbitrary and $|t_0-t_D| \leq C_0R$.}
\end{lemma}
 We will also need the analogue of this for blue waves. 

\begin{proof}
Using H\"{o}lder inequality, it is enough to prove this for $p=q=2$. To see (\ref{eq3.9}), observe that
\[\|(1-P_D)\phi\|_{L^{\infty}(Q(x_D,t_D;C^{-1}r))} \lesssim r^{C-N}E(\phi)^{1/2}\]
by  (\ref{eq.fi}), (\ref{eq.k}), (\ref{eq3.5}). Then by (\ref{ce}) we get the desired result.
To prove (\ref{eq3.10}) we observe that
\[ \| P_D\phi \|_{L^{\infty}(Q(x_0,t_0;R) \setminus C^R(x_D,t_D;Cr+R^{1/N}))} \lesssim R^{C-N}E(\phi)^{1/2} \]
by (\ref{eq.fi}), (\ref{eq.k}), (\ref{eq3.4}). Then by (\ref{ce}) we obtain our result. Analogues for blue waves are obtained similarly without any loss. 
\end{proof}

Using transversality we can prove better $L^2$ estimates than (\ref{ce}). Next lemma which is proven in \cite{T1} shows this.

\begin{lemma}[Lemma 13.1 in \cite{T1}]\label{l3.3} Let $\phi$ be a red wave. Then for any $(x_0,t_0)\in \textbf{R}^{n+1}$ and $R \gtrsim 1$ we have
\[\|\phi\|_{L^2(C^{B}(x_0,t_0;R))} \lesssim R^{1/2}E(\phi)^{1/2}.\]
\end{lemma}
By time reversal we of course have an analogue of this for blue waves. Hence we have the following corollary.

\begin{corollary}[Corollary 13.2 in \cite{T1}]\label{c3.1} Let $\phi$ be a  red wave and $\psi $ a blue wave. Let $R > r \gg 1$, $(x_0,t_0)\in \textbf{R}^{n+1}$, and $Q_R$ be any cube of side-length R. Then 
\[\|\phi \psi\|_{L^{1}(C^{P}(x_0,t_0;r)\cap Q_R)} \lesssim r^{1/2}R^{1/2}E(\phi)^{1/2}E(\psi)^{1/2}.\]
 
\end{corollary}

Finally we give a lemma that will be used in section \ref{s5}.

\begin{lemma}\label{lx} Let $R>0$, $0<c \leq 2^{-C}$, and $Q_R$ be a cube of side-length $R$.  Let $F$ be any function essentially bounded on $C_0Q_R$.   Then there exists a cube $Q$ of side-length $CR$ contained in $C^2Q_R$ such that
\[\|F\|_{L^q_tL^p_x(Q_R)}\leq (1+Cc)\|F\|_{L^q_tL^p_x(I^{c,C_0}(Q))}\]
\end{lemma}

\begin{proof}
We first prove this for $L^1$ then by duality arguments extend it to $L^q_tL^p_x$. Let $G $ be integrable on $C_0Q_R$. By pigeonhole principle it suffices to prove
\[\|G\|_{L^1(Q_R)} \leq \frac{1}{|Q_R|}\int_{Q_R}(1+Cc)\|G\|_{L^1(I^{c,C_0}(Q(x_0,t_0;CR))\cap Q_R)} dx_0dt_0.\]
 Then applying Fubini's theorem we have
\[
\int_{Q_R}\|G\|_{L^1(I^{c,C_0}(Q)\cap Q_R)} dx_0dt_0  
=  \int_{Q_R} |G(x,t)||I^{c,C_0}(Q)\cap Q_R|dxdt.
\]
But we have
\[ |Q(x_0,t_0;CR)\setminus I^{c,C_0}(Q(x_0,t_0;CR))| \lesssim c|Q(x_0,t_0;CR)| \]
hence
\[|Q_R|\leq (1+Cc)|I^{c,C_0}(Q(x_0,t_0;CR)) \cap Q_R|.\]
from which the result for $L^1$ follows. Now observe that it suffices to prove our lemma for  $\|F\|_{L^q_tL^p_x(Q_R)}=1$. We have, by duality, a function $A$ such that $\|A\|_{L^{q'}_tL^{p'}_x(Q_R)}=1$ and 
\[ \int_{Q_R}|F(x,t)|A(x,t)dxdt=1.\]
But notice that by our result for $L^1$ functions we have 
\begin{align*}
1=\big|\int_{Q_R}|F(x,t)|A(x,t)dxdt \big| &\leq \|FA\|_{L^1(Q_R)} \\
& \leq (1+Cc)\|FA\|_{L^1(I^{c,C_0}(Q))}.
\end{align*} 
Then our result follows from the H\"{o}lder inequality.   

\end{proof}
 
 \section{The key proposition}\label{s4}

 In this section we will state and prove the key proposition that will be used in the next section. First we import a proposition from \cite{T1} on which we do not need to make any change. As stated there, we have an analogue of this for blue waves.
\begin{proposition}[Proposition 15.1 in \cite{T1}]\label{p4.1} Let $R \geq C_02^{C_1}$, $0<c \leq 2^{-C_0}$. Let $Q$ be a spacetime 
cube of side-length R. Let $\phi$ be a red wave such that $margin(\phi) \gtrsim R^{-1/2}$, and let $\psi$ be a blue wave. Then there exists a red wave table $\Phi=\Phi_c(\phi,\psi;Q)$ of depth $C_0$ on Q such that the following properties hold.
\begin{equation}\label{eq4.1}
margin(\Phi) \geq  margin(\phi)-CR^{-1/2}.  
\end{equation}
$ [\Phi]_{C_0} \ approximates \ \phi: $
\begin{equation}\label{eq4.2} \left\| (|\phi|-[\Phi]_{C_0}) \psi \right\|_{L^2(I^{c,C_0}(Q))} \lesssim c^{-C}R^{(1-n)/4}E(\phi)^{1/2}E(\psi)^{1/2} . 
\end{equation}
$Bessel \ inequality:$
\begin{equation}\label{eq4.3}
E(\Phi) \leq (1+Cc)E(\phi) .
\end{equation}
$Persistence \ of \ non-concentration: For \ any \ r \gtrsim R^{(1/2+1/N)} \ we \ have $ 
\begin{equation}\label{eq4.4}
E_{r(1-C_0r^{-1/2N}),C_0Q}(\Phi,\psi) \leq (1+Cc)E_{r,C_0Q}(\phi,\psi). 
\end{equation}
\end{proposition}

Now we state and prove our key proposition.
\begin{proposition}\label{p4.2} Let $R \geq C_0 2^{C_1},0<c \leq 2^{-C_0} $, and let $\phi, \psi$ be respectively red and blue waves  which obey the energy normalization and the relaxed margin requirement
\begin{equation}\label{rm}
margin(\phi), margin(\psi) \geq 1/100-2(1/R)^{1/N}.
\end{equation}
 Then for any cube $Q$ of side-length $CR$, we can find on $Q$ a red wave table $\Phi$ of   depth $C_0$  
 and a blue wave table $\Psi$ of depth $C_0$  such that the following properties hold. \\
We have the margin estimate
 \begin{equation}\label{eq4.5}
 margin(\Phi),margin(\Psi) \geq 1/100- 3(1/R)^{1/N}.
 \end{equation}
 We have the energy estimate
 \begin{equation}\label{eq4.6}
 E(\Phi),E(\Psi)\leq 1+Cc.      
 \end{equation}
The following inequality holds
\begin{equation}\label{eq4.7}
 \left\|\phi \psi \right\|_{L^q_t L^p_x(I^{c,C_0}(Q))} \leq \left\|  [\Phi]_{C_0}[\Psi]_{C_0}\right\|_{L^q_tL^p_x(I^{c,C_0}(Q))} + c^{-C} .  
 \end{equation}
 If $r>1$ then for any cone $C^{purple}(x_0,t_0;r)$ we have
\begin{equation}\label{eq4.8} 
\begin{aligned}
\left\|\phi \psi \right\|_{L^q_tL^p_x(I^{c,C_0}(Q)\cap C^{P}(x_0,t_0;r))} &\leq \left\|  [\Phi]_{C_0}[\Psi]_{C_0}\right\|_{L^q_tL^p_x(I^{c,C_0}(Q)\cap C^{P}(x_0,t_0;r))}\\ &+ c^{-C}(1+R/r)^{-\epsilon/4} .
\end{aligned}
\end{equation}
We have the persistence of non-concentration: for all $r \gtrsim R^{1/2+3/N}$
\begin{equation}\label{eq4.9}
   E_{r(1-C_0(r)^{-1/3N}), C_0Q}(\Phi,\Psi) \leq E_{r,C_0Q}(\phi,\psi)+Cc. 
\end{equation} 
 \end{proposition}

\begin{proof}  Define $\Phi:=\Phi_c(\phi,\psi,c)$ as in the Proposition 4.1. Then
 \begin{equation}
 \begin{aligned}
 margin(\Phi)\geq margin(\phi)-CR^{-1/2} &\geq 1/100-2R^{-1/N}-CR^{-1/2} \\ &\geq 1/100-3(1/R)^{1/N}.
 \end{aligned}
 \end{equation}
 Hence we have the margin requirement on $\Phi$. Energy estimate directly follows from the definition of $\Phi$. Let $\Psi:=\Psi_c(\Phi,\psi,c)$. Energy and margin requirements follow from time reversal. 
 
 We now prove (\ref{eq4.7}).  By (\ref{eq4.2}) we have 
\begin{equation}\label{eq4.10}
 \|(|\phi|-[\Phi]_{C_0})\psi\|_{L^2(I^{c,C_0}(Q))} \lesssim c^{-C}R^{\frac{1-n}{4}}.
\end{equation}
 On the other hand by (\ref{eq3.12}), (\ref{eq3.2}) and (\ref{eq4.6})
 \[ \|\phi \psi\|_{L^1(I^{c,C_0}(Q))}, \\  \|[\Phi]_{C_0} \psi\|_{L^1(I^{c,C_0}(Q))} \lesssim R. \]
 So by triangle inequality we have
\begin{equation}\label{eq4.11}
\|(|\phi|-[\Phi]_{C_0})\psi\|_{L^1(I^{c,C_0}(Q))} \lesssim R.
\end{equation}
 By H\"{o}lder (\ref{eq4.10}) gives
\begin{equation}\label{n1}
 \|(|\phi|-[\Phi]_{C_0})\psi\|_{L^q_tL^2_x(I^{c,C_0}(Q))} \lesssim c^{-C}R^{(\frac{1-n}{4}+\frac{2-q}{2q})}.
 \end{equation}
 To handle $L^1$ case, observe that using (\ref{eq2.1}) together with (\ref{eq3.2}), (\ref{eq4.6}) and the triangle inequality one obtains
\[ \|(|\psi|-[\Phi]_{C_0})\psi\|_{L_t^{\infty}L_x^{1}(I^{c,C_0}(Q))} \lesssim 1 . \]
 We interpolate this last inequality with (\ref{eq4.11}) to get,  
 \[\|(|\phi|-[\Phi]_{C_0})\psi\|_{L^q_tL_x^1(I^{c,C_0}(Q))} \lesssim R^{1/q} .\]
Then by interpolating the last one with (\ref{n1}) we obtain
\begin{equation}\label{eq4.12}
\|(|\phi|-[\Phi]_{C_0})\psi\|_{L^q_tL^p_x(I^{c,C_0}(Q))} \lesssim c^{-C}R^{(\frac{1-n}{4}+\frac{2-q}{2q})(2-\frac{2}{p})+\frac{1}{q}(\frac{2}{p}-1)} =c^{-C}.
\end{equation}
By the analogue of (\ref{eq4.2}) for blue waves, (\ref{eq3.2}) and (\ref{eq4.6})   we have
\begin{equation}\label{eq4.13}
\begin{aligned}
\|(|\psi|-[\Psi]_{C_0})[\Phi]_{C_0}\|_{L^2(I^{c,C_0}(Q))} &\leq \|(|\psi|-[\Psi]_{C_0})\Phi\|_{L^2(I^{c,C_0}(Q))}\\  &\lesssim c^{-C}R^{\frac{1-n}{4}}.
\end{aligned}
\end{equation}
For $L^1$ case by (\ref{eq3.12}), (\ref{eq3.2}) and (\ref{eq4.6})  we have  
\[ \| \psi [\Phi]_{C_0} \|_{L^1(I^{c,C_0}(Q))}, \\ \|[\Phi]_{C_0}[\Psi]_{C_0}\|_{L^1(I^{c,C_0}(Q))} \lesssim R \] 
so by the triangle inequality
\begin{equation}\label{eq4.14}
\|(|\psi|-[\Psi]_{C_0})[\Phi]_{C_0}\|_{L^1(I^{c,C_0}(Q))} \lesssim R.
\end{equation}
Then we apply H\"{o}lder and interpolation to (\ref{eq4.13}) and (\ref{eq4.14}) exactly as we did to (\ref{eq4.10}) and (\ref{eq4.11}) to get
\begin{equation}\label{eq4.15}
\|(|\psi|-[\Psi]_{C_0})[\Phi]_{C_0}\|_{L^q_tL^p_x(I^{c,C_0}(Q))} \lesssim c^{-C}.
\end{equation}
The triangle inequality together with (\ref{eq4.15}) and (\ref{eq4.12}) gives (\ref{eq4.7}).

 We will apply the same process to prove (\ref{eq4.8}). Let
 \[ \Omega:=I^{c,C_0}(Q) \cap C^{P}(x_0,t_0;r)\]
 We shall assume $R>r$ since otherwise (\ref{eq4.12}), (\ref{eq4.15}) combined with the triangle inequality and the fact that $\Omega \subseteq I^{c,C_0}(Q)$ gives (\ref{eq4.8}).
 In $L^2$ case using the triangle inequality, (\ref{eq4.10}), (\ref{eq4.13}), and the fact that $\Omega \subseteq I^{c,C_0}(Q)$ we have
 \begin{equation}\label{eq4.16}\| \phi\psi-[\Phi]_{C_0}[\Psi]_{C_0}\|_{L^2(\Omega)} \lesssim c^{-C} R^{\frac{1-n}{4}}.
 \end{equation}
 In $L^1$ case, using Corollary \ref{c3.1}, (\ref{eq3.2}) and (\ref{eq4.6}) we obtain
 \[\|\phi\psi\|_{L^1(\Omega)}, \|[\Phi]_{C_0}[\Psi]_{C_0} \|_{L^1(\Omega)} \lesssim    (r/R)^{1/2} R.\]
 Hence by the triangle inequality
 \begin{equation}\label{eq4.17}
 \| \phi\psi-[\Phi]_{C_0}[\Psi]_{C_0}\|_{L^1(\Omega)} \lesssim (r/R)^{1/2}R .
 \end{equation}
 Apply H\"{o}lder to (\ref{eq4.16}) as above to get
 \begin{equation}  \| \phi\psi-[\Phi]_{C_0}[\Psi]_{C_0}\|_{L^q_tL^2_x(\Omega)} \lesssim c^{-C} R^{\frac{1-n}{4}+\frac{1}{q}- \frac{1}{2}} .   
 \end{equation}
 On the other hand (\ref{eq2.1}) together with (\ref{eq3.2}), (\ref{eq4.6}) and the triangle inequality yields
 \[ \| \phi\psi-[\Psi]_{C_0}[\Psi]_{C_0}\|_{L_t^{\infty}L_x^{1}(\Omega)} \lesssim 1 . \]
  interpolating this with (\ref{eq4.17}) we obtain.
 \begin{equation}\label{eq4.18}
  \| \phi\psi-[\Phi]_{C_0}[\Psi]_{C_0}\|_{L^q_tL^1_x(\Omega)} \lesssim (r/R)^{1/2q}R^{1/q} .
  \end{equation}
 Then interpolating (\ref{eq4.17}) with (\ref{eq4.18}) gives
 \[ \| \phi\psi-[\Phi]_{C_0}[\Psi]_{C_0}\|_{L^q_tL^p_x(\Omega)} \lesssim c^{-C}(r/R)^{\epsilon/4}. \]
 By the triangle inequality we get (\ref{eq4.8}).
 
 Now it remains to prove (\ref{eq4.9}). Fix $r \gtrsim R^{1/2+3/N}$, and pick $\rho$ such that
 $\rho(1-C_0\rho^{-1/2N})=r(1-C_0r^{-1/3N})$. Clearly such a $\rho$ value exists, furthermore it satisfies $\rho \gtrsim R^{1/2+1/N}$, and $\rho \leq r(1-C_0r^{-1/2N}) $.  Then using monotonicity of energy concentration, (\ref{eq4.9}) and its analogue for blue waves we have
 \begin{align*}
 E_{r(1-C_0r^{-1/3N}),C_0Q}(\Phi,\Psi) &=E_{\rho(1-C_0\rho^{-1/2N}),C_0Q}(\Phi,\Psi)\\ &\leq (1+Cc)E_{\rho,C_0Q}(\Phi,\psi)\\  &\leq (1+Cc)E_{r(1-r^{-1/2N}),C_0Q}(\Phi,\psi) \\ &\leq(1+Cc)E_{r,C_0Q}(\phi,\psi).
 \end{align*}
 from which our result follows by the energy normalization. 
 \end{proof}

\section{Proof of Theorem 1.1}\label{s5} 
 
 At the end of section \ref{s2} we localized to cubes, and then in section \ref{s4} to sub-cubes.  The following proposition completes the first paragraph of the sketch of the proof given in section \ref{s2}.

 \begin{proposition}\label{p5.2}Suppose $R \geq 2 C_0 2^{C_1}$ and $0<c \leq 2^ {C_0}$ and $\phi, \psi$ respectively red and blue waves satisfying the energy normalization and the relaxed margin requirement (\ref{rm}).
Then for any cube $Q_R$ of side length $R$ one has
\begin{equation}\label{it} \left\| \phi \psi\right\|_{L^q_t L^p_x (Q)}  \leq (1+Cc)\overline{A}(R/2)E(\phi)^{1/2}E(\psi)^{1/2} + c^{-C} . \end{equation}
\end{proposition}

\begin{proof}
 Using Lemma \ref{lx} with $F:=\phi\psi$ we can find a cube Q of side-length CR inside $C^2Q_R$ such that
\[\|\phi\psi\|_{L^q_tL^p_x(Q_R)} \leq (1+Cc)\|\phi\psi\|_{L^q_tL^p_x(I^{c,c_0}(Q))}.\]
Let $\Phi,\Psi$ be as in Proposition \ref{p4.2}. Then by  (\ref{eq4.7}), we have
\begin{equation}\label{eq5.1}
\|\phi\psi\|_{L^q_tL^p_x(Q_R)} \leq (1+Cc)\|[\Phi]_{C_0}[\Psi]_{C_0}\|_{L^q_tL^p_x(I^{c,C_0}(Q))}+ c^{-C}.
\end{equation}
Applying the triangle inequality we get
\[ \|[\Phi]_{C_0}[\Psi]_{C_0}\|_{L^q_tL^p_x(Q)} \leq \sum_{q \in K_{C_0}(Q)} \|\Phi^{(q)}\Psi^{(q)}\|_{L^q_tL^p_x(q)}.\]
Then (\ref{eq4.5}) combined with Definition \ref{d2.1} gives
\[\|[\Phi]_{C_0}[\Psi]_{C_0}\|_{L^q_tL^p_x(Q)} \leq  A(2^{-C_0}R) \sum_{q \in K_{C_0}(Q)} E(\Phi^{(q)})^{1/2}E(\Psi^{(q)})^{1/2}\]
Cauchy-Schwarz combined with (\ref{eq3.1}) and (\ref{eq4.6}) we obtain
\[ \|[\Phi]_{C_0}[\Psi]_{C_0}\|_{L^q_tL^p_x(Q)} \leq (1+Cc)\overline{A}(R/2).\]
which, inserted to (\ref{eq5.1}), gives the desired result.
\end{proof}

Iterating (\ref{it})  and using a globalization lemma gives non-endpoint/non-endline results; see \cite{LV1}, \cite{T1}. For our purposes, set $c=(1/R)^{1/N}$,  iterate (\ref{it}) and use (\ref{pb}) when $R \approx C_02^{C_1}$ to obtain
\begin{equation}\label{cn}
A(R)\lesssim 2^{CC_1}R^{C/N}
\end{equation}  
This last inequality proves (\ref{eq2.2}) for all $2C_02^{C_1} \leq R \leq  C_02^{NC_1}$. For larger $R$ we shall  introduce the notion of energy concentration. 

\begin{definition}\label{d5.1}
Let $r>0$. Let $Q$ be a space-time cube of side-length $R$, let $\phi$ a red wave, and $\psi$ a blue wave. The the energy concentration  $E_{r,Q}$ is defined to be  
 \[E_{r,Q}(\phi,\psi):= \max \left\{ \frac{1}{2} E(\phi)^{1/2} E(\psi) ^{1/2},\sup_D \left\| \phi \right\|_{L^2(D)} \left\| \psi \right\|_{L^2(D)}   \right\} \]
 where supremum is taken over all disks of radius $r$ whose time coordinate is inside the life-span of $Q$.
 \end{definition} 
The next definition gives a variant of $A(R)$ 
 which is sensitive to energy concentration, allows one to do induction on scales in cone neighborhoods successfully, and can be related to $A(R)$. With this variant at hand one first bounds $A(R)$ by this variant with some gain, then handles concentrated and non-concentrated cases separately.

 \begin{definition}\label{d5.2} Let  $R \geq 2^{NC_1/2}$ and $r,r'>0$. Then $A(R,r,r')$ is defined to be the best constant for which the inequality
 \[  \left\| \phi \psi \right\|_{L^q_tL^p_x(Q_R \cap C^{P}(x_0,t_0;r'))} \leq A(R,r,r') (E(\phi)^{1/2} E(\psi) ^{1/2})^{1/q} E_{r,C_0Q_R}(\phi,\psi)^{1/q'} \]
 holds for all spacetime cubes $Q_R$ of side-length $R$, all $(x_0,t_0) \in \textbf{R}^{n+1}$, red waves $\phi$ and  blue waves $\psi$  that obey the strict margin requirement (\ref{smr}).
 \end{definition}

While using this definition to do induction on scales in cone neighborhoods, and to bound $A(R)$ with $A(R,r,r')$ one needs to use the following lemma instead of the triangle inequality to make some exponential gain.

\begin{lemma}\label{l2.1}
Let $f_1,f_2...f_k$ be a finite collection of functions such that $f_j:\textbf{R}^{n+1}\rightarrow H $, and $f_j \in L^q_tL^p_x(\textbf{R}^{n+1}), \ 1\leq j \leq k$ where H is a finite dimensional complex Hilbert space. If $q<p$ and supports of these functions are mutually disjoint then
\[ \big\| \sum_{j=1}^k f_j \big \|_{L^q_tL^p_x}^q \leq \sum_{j=1^k} \left\|f_j \right\|_{L^q_tL^p_x}^q . \]

\end{lemma}

\begin{proof}
 We first exploit disjointness of supports, and then concavity:
\begin{align*}
\int (\int |\sum_{j=1}^k f_j(x,t)|^p dx )^{q/p}dt &= \int( \sum_{j=1}^k \int |f_j(x,t)|^p dx )^{q/p}  dt\\                                                                         &\leq \sum_{j=1}^k \int (\int |f_j(x,t)|^pdx )^{q/p}dt.           
\end{align*}

\end{proof}

This lemma shows that the following fact about $L^p$ norms extends partially to mixed norms.
Let $f_1,f_2...f_k$ be a finite collection of functions such that $f_j:\textbf{R}^{n+1}\rightarrow H $, and $f_j \in L^p(\textbf{R}^{n+1}),\  1\leq j \leq k$ where H is a finite dimensional complex Hilbert space. If supports of these functions are mutually disjoint then 
 
 \[\big\| \sum_{j=1}^k f_j   \big\|_p^p = \sum_{j=1}^k \left\|f_j \right\|_p^p .\]

Our lemma, of course, is not so good as the property of $L^p$ norms given above, but will do in our case.
We now  exploit non-concentration and  relate $A(R)$ to $A(R,r,r')$ with some gain.

\begin{proposition}\label{p5.3} 
 Let $R\geq 2^{NC_1}$. Then we have
 \[ A(R) \leq (1-C_0^{-C}) \sup_{\stackrel{2^{NC_1} \leq \widetilde{R} \leq R}{\widetilde{R}^{1/2+4/N} \leq r}}  A(\widetilde{R},r,C_0(1+r))  +2^{CC_1}. \] 
\end{proposition}

We shall need the following lemma in the proof.

 \begin{lemma}\label{l5.1} 
 Let $R \geq 2^{NC_1}$ and $2^{ NC_1/2} \leq r \leq R^{1/2+4/N}$. Let $D=D(x_D,t_D;C_0^{1/2}r)$ be a disk. Let $\phi,\psi$ be respectively red and blue waves  with $margin(\phi),margin(\psi)\geq 1/200$. Then we have  
   \[ \left\| (P_{D}\phi) \psi \right \|_{L^q_tL^p_x(Q^{ann}(x_0,t_0;R,2R))} \lesssim   R^{-1/C}E(\phi)^{1/2}E(\psi)^{1/2} \] 
 \[  \left\| \phi P_{D} \psi \right\|_{L^q_tL^p_x(Q^{ann}(x_0,t_0;R,2R) )}  \lesssim R^{-1/C}E(\phi)^{1/2}E(\psi)^{1/2}. \]
 \end{lemma} 
 We first prove the lemma, then the proposition. 

\begin{proof} By translation invariance we can take  $(x_0,t_0)=(0,0)$. First we consider $\left\| (P_{D}\phi) \psi \right\|_{L^q_tL^p_x(Q^{ann}(x_0,t_0;R,2R)  )} .$  By using H\"{o}lder and interpolation as in the proof of Proposition 4.1 it suffices to prove
  \[ \left\| (P_{D}\phi) \psi \right\|_{L^1(Q^{ann}(x_0,t_0;R,2R))} \lesssim R^{C/N}R^{3/4}, \]
  \[\left\| (P_{D}\phi) \psi \right\|_{L^2(Q^{ann}(x_0,t_0;R,2R))} \lesssim R^{C/N}R^{\frac{1-n}{4}}.\]
  But frequency of $\psi$ plays no role in the proof given in \cite{T1} and so the same proof works. For $\left\| \phi (P_{D} \psi )\right\|_{L^q_tL^p_x(Q^{ann}(x_0,t_0;R,2R))}$ since we have no difference between frequencies of $\phi$ and $\psi$, by time reversal we get the same result without any loss.
\end{proof}

 \begin{proof} Let $Q_R$ be a spacetime cube of side-length R. Let $\phi,\psi$ be respectively red and blue waves  with strict margin requirement (\ref{smr}) and the energy normalization. Clearly it suffices to prove
 
  \[ \left\| \phi \psi \right\|_{L^q_tL^p_x(Q_R)} \leq (1-C_0^{-C}) \sup_{\stackrel{2^{NC_1} \leq \widetilde{R} \leq R}{\widetilde{R}^{1/2+4/N} \leq r} }  A(\widetilde{R}, r,C_0(1+r)) + 2^{CC_1} .  \]
  We may of course assume that $\left\| \phi\psi\right\|_ {L^q_tL^p_x(Q_R)}\approx A(R)$ and that $A(R) \geq 2^{CC_1}$. Let $0< \delta < 1/4$ be a small number to be specified later, and let $r$ be the supremum of all radii $r \geq 2^{NC_1(1/2+4/N)}$ such that $E_{r,C_0Q_R} (\phi,\psi) \leq 1-\delta $ or $r=2^{NC_1(1/2+4/N)}$ if no such radius exists.  Let $D:=D(x_0,t_0;r)$ be a disk with $t_D$ in the lifespan of $C_0Q_R$, and  
  \begin{equation}\label{eq5.2}
  min(\left\|\phi \right\|_{L^2(D)}, \left\|\psi \right\|_{L^2(D)}) \geq 1-2\delta.
  \end{equation}
   Such a disk clearly exists by the definition of r. 
   Let $D'=C_0^{1/2}D$ and $ \Omega = Q_R \cap C^{P}(x_0,t_0;C_0(1+r))$. Let $\phi=(1-P_{D'})\phi+ P_{D'}\phi $, and $\psi=(1-P_{D'})\psi+ P_{D'}\psi$.
   
   We have two cases: $r>R^{1/2+4/N}$ or $r\leq R^{1/2+4/N}$. So first assume $r>R^{1/2+4/N}$.    Then by  (\ref{eq3.7}),  (\ref{eq3.8}) and (\ref{eq5.2}) we have 
   \[E((1-P_{D'})\phi),E((1-P_{D'})\psi) \lesssim \delta +C_0^{-C} .\] 
   Thus by (\ref{eq.r'}) one has 
   \begin{equation}\label{eq5.3}
    \left\|(1-P_{D'})\phi(1-P_{D'})\psi \right\|_{L^q_tL^p_x(Q_R)} \lesssim (\delta+C_0^{-C})A(R)  .
   \end{equation}
   By (\ref{eq3.10}) and its analogue for blue waves we have  
   \[ \left\| (P_{D'}\phi) \psi \right\|_{L^q_tL^p_x(Q_R\setminus \Omega)},  \left\| (1-P_{D'})\phi P_{D'}\psi \right\|_{L^q_tL^p_x(Q_R\setminus \Omega)}  \lesssim C_0^{-C}.\]
   Then by the triangle inequality and our assumptions on $A(R)$ at the beginning of the proof we have 
   \[
    \left\| \phi \psi \right\|_{L^q_tL^p_x(Q_R\setminus \Omega)} \lesssim (\delta+C_0^{-C})A(R) \lesssim (\delta+C_0^{-C})\left\| \phi \psi \right\|_{L^q_t L^p_x(Q_R)} .
   \]
Here we will use Lemma \ref{l2.1} instead of directly applying triangle inequality. This is where we cede the uppermost endpoint $(n+1/n-1,1)$ when $n \geq 3.$ 
 \begin{align*} \| \phi \psi \|^q_{L^q_tL^p_x(Q_R)}  &\leq   
 \| \phi \psi \|^q_{L^q_tL^p_x(Q_R \setminus \Omega)} +   \| \phi \psi \|^q_{L^q_tL^p_x( \Omega)} \\ &\leq C(\delta+C_0^{-C})^q \| \phi \psi \|^q_{L^q_tL^p_x(Q_R)} +  \| \phi \psi \|^q_{L^q_tL^p_x( \Omega)}  .
 \end{align*} 
  Hence,
  \[ \left\| \phi \psi \right\|_{L^q_tL^p_x( \Omega)} \geq (1-C(\delta+C_0^{-C})^q)^{1/q} \left\| \phi \psi \right\|_{L^q_tL^p_x(Q_R)}.  \]  
On the other hand, by our assumption $r>R^{1/2+4/N}$ and by the  definition of $r$ we have
  \[\left\| \phi \psi \right\|_{L^q_tL^p_x( \Omega)} \leq A(R,r,C_0(1+r))(1-\delta)^{1/q'} .  \]
  But then setting $\delta=C_0^{-C}$ one obtains the desired estimate.
  
  Now we handle the second case. Define $\widetilde{R}:= r^{\frac{1}{1/2+4/N}}$. Thus 
$2^{NC_1}\leq \widetilde{R} \leq R$ and $r\geq \widetilde{R}^{1/2+4/N}$.  If  $\widetilde{R} > 2^{NC_1}$ then by Definition \ref{d5.2} one has
\[ \| \phi \psi \|_{L^q_tL^p_x(Q(x_0,t_0;\widetilde{R}) \cap \Omega)} \leq A(\widetilde{R},r,C_0(1+r))(1-\delta)^{1/q'}. \]
If $R=2^{NC_1}$ then we have by (\ref{cn})   
\[\| \phi \psi \|_{L^q_tL^p_x(Q(x_0,t_0;\widetilde{R}) \cap \Omega)} \leq 2^{CC_1}.\]
  Note that with this definition of $\widetilde{R}$, we can obtain
  \[
    \| \phi \psi \|_{L^q_tL^p_x(Q(x_0,t_0;\widetilde{R}) \setminus \Omega)} \lesssim (\delta+C_0^{-C})\| \phi \psi \|_{L^q_tL^p_x(Q_R)} \]
  by the same arguments as above. 
  Hence if we can show that 
  \[ \left\| \phi \psi \right\|_{L^q_tL^p_x(Q_R \setminus Q(x_0,t_0;\widetilde{R}) )} \lesssim (\delta+C_0^{-C})A(R)   \]
  then we apply the Lemma \ref{l2.1} as we did above and obtain the desired result. To show this together with (\ref{eq5.3}) we need the estimates
  \begin{align*} 
  \left\| (P_{D'}\phi) \psi \right\|_{L^q_tL^p_x(Q_R \setminus Q(x_0,t_0;\widetilde{R})) } &\lesssim  (\delta+C_0^{-C})A(R), \\ \left\| ((1-P_{D'})\phi) P_{D'} \psi \right\|_{L^q_tL^p_x(Q_R \setminus Q(x_0,t_0;\widetilde{R}) )}  &\lesssim (\delta+C_0^{-C})A(R). 
  \end{align*} 
  But by a dyadic decomposition these would follow Lemma \ref{l5.1}.
  \end{proof}

Now it remains to bound  $A(R,r,r')$ by $A(R)$. This we will do in two steps: the non-concentrated case and the concentrated case. First we deal with the non-concentrated case.

\begin{proposition}\label{p5.5}
Let $R\geq 2^{NC_1/2}$, $r \geq C_0^C R$, $r'>0$ and $0<c \leq 2^{-C_0}$. Then we have
 \[ A(R,r,r')\leq (1+Cc)\overline{A}(R) + c^{-C}.\]
 \end{proposition}
  
  \begin{proof} Let $\phi,\psi$ be respectively red and blue waves  that satisfy the strict margin requirement (\ref{smr}), and the energy normalization. Then it is enough to prove that 
  \[ \left\| \phi \psi \right\|_{L^q_tL^p_x(Q_R)} \leq E_{r,C_0Q_R}(\phi,\psi)^{1/q'}(1+Cc)\overline{A}(R) +c^{-C}\]
  where $Q_R$ is an arbitrary cube of side-length $R$. Let $D:=D(x_{Q_R},t_{Q_R};r/2)$ where $(x_{Q_R},t_{Q_R})$ is the center of $Q_R$. We will decompose our waves: $\phi= (1-P_D)\phi+P_D\phi$, $\psi=(1-P_D)\psi+P_D\psi$. By Lemma \ref{l3.1}  $P_D\phi$, $P_D\psi$ satisfy relaxed margin requirement (\ref{rm}) and the energy estimate
  \[E(P_D\phi)^{1/2}E(P_D\psi)^{1/2} \leq E_{r,C_0Q_R}(\phi,\psi)+CR^{C-N/2}.\]
  So we can apply Proposition \ref{p5.2} to get
  \[\left\| (P_D\phi) (P_D\psi )\right\|_{L^q_tL^p_x(Q_R)} \leq (1+Cc)(E_{r,C_0Q_R}(\phi,\psi)+CR^{C-N/2})\overline{A}(R) + c^{-C}.\]
  Using a trivial polynomial bound on  $\overline{A}(R)$ we absorb $CR^{C-N}$ into $c^{-C}$. Hence we will be done if we can show that
 
  \[ \left\| ((1-P_D)\phi) \psi \right\|_{L^q_tL^p_x(Q_R)}, \left\| (P_D\phi) (1-P_D)\psi \right\|_{L^q_tL^p_x(Q_R)} \leq c^{-C}.\]
  Both of these follow from (\ref{eq3.9}) and its analogue for blue waves. 
\end{proof}

 We now turn to concentrated case.

\begin{proposition}\label{p5.6} Let $R\geq C_02^{NC_1/2}$ and $C_0^C R \geq r > R^{1/2+3/N}$. Then we have 
 \[A(R,r,r')\leq (1+Cc)A(R/C_0,r(1-Cr^{-1/3N}),r') + c^{-C}(1+\frac{R}{r'})^{-\epsilon/4}   \]
 for any $0< c \leq 2^{-C_0}$.
 
\end{proposition}
 
 \begin{proof}Let $Q_R$ be a  spacetime cube of side-length R,  $(x_0,t_0)$ be an element of $\textbf{R}^{n+1}$, $\phi$ and $\psi$ respectively red and blue waves   that obey the strict margin requirement (\ref{smr}). Then it is enough to prove that 
 \begin{align*}
  \left\| \phi \psi \right\|_{L^q_tL^p_x(Q_R \cap C^{P}(x_0,t_0;r'))} &\leq (1+Cc)A(R/C_0,\widetilde{r},r')E_{r,C_0Q_R}(\phi,\psi)^{1/q'}\\ &+ c^{-C}(1+\frac{R}{r})^{-\epsilon/4} 
  \end{align*}
 where  $\widetilde{r} = r(1-Cr^{-1/3N})$ since $E_{r,C_0Q_R}(\phi,\psi) \approx 1$.
 
 We will perform some reductions. By Lemma \ref{lx} applied to $\phi \psi \chi_{C^{P}(x_0,t_0;r)}$ there is a cube $Q$ of side-length $CR$ contained in $C^2Q_R$ such that
 \[\| \phi \psi \|_{L^q_tL^p_x(Q_R \cap C^P(x_0,t_0;r'))} \leq  (1+Cc)\left\| \phi \psi \right\|_{L^q_tL^p_x(I^{c,C_0}(Q) \cap C^P(x_0,t_0;r'))}.\] 
 Applying Proposition \ref{p4.2} we reduce to showing 
 \begin{align*}
 \| [\Phi]_{C_0}[\Psi]_{C_0} \|_{L^q_tL^p_x(I^{c,C_0}(Q) \cap C^{P}(x_0,t_0;r'))} &\leq (1+Cc)A(R/C_0,\widetilde{r},r')E_{r,C_0Q_R}(\phi,\psi)^{1/q'}\\ &+ c^{-C}R^{C-N/2}. 
  \end{align*}
  Using (\ref{eq4.9}) it suffices to prove that
 \[ \| [\Phi]_{C_0}[\Psi]_{C_0} \|_{L^q_tL^p_x(I^{c,C_0}(Q) \cap C^{P}(x_0,t_0;r'))} \leq (1+Cc)A(R/C_0,\widetilde{r},r')E_{\widetilde{r},C_0Q}(\Phi,\Psi)^{1/q'}.\] 
 Using Lemma \ref{l2.1} it is enough to prove that
 \[ \sum_{q \in \textit{K}_{C_0}(Q)} \| \Phi^{(q)} \Psi^{(q)} \|^q_{L^q_tL^p_x(q\cap C^P(x_0,t_0;r'))} \leq (1+Cc)A(R/C_0,\widetilde{r},r')^q E_{\widetilde{r},C_0Q}(\Phi,\Psi)^{q/q'}. \]
 
  The observation $E_{\widetilde{r},C_0Q}(\Phi^{(q)},\Psi^{(q)}) \leq E_{\widetilde{r},C_0Q}(\Phi,\Psi)$ together with Definition \ref{d5.2} yield
  \[ \| \Phi^{(q)} \Psi^{(q)} \|^q_{L^q_tL^p_x(q\cap C^P(x_0,t_0;r'))} \leq A(R/C_0,\widetilde{r}, r')^{q} E(\Phi^{(q)})^{1/2}E(\Psi^{(q)})^{1/2} E_{\widetilde{r},C_0Q}(\Phi,\Psi)^{q/q'} . \]
   But then summing up followed by Cauchy-Schwarz and (\ref{eq4.6}) will yield the desired result.
 
 \end{proof}
  
 We combine these two propositions to obtain the following corollary.
 
 \begin{corollary}
  Let $R\geq 2^{NC_1}$ and $r\geq R^{1/2+4/N}$. Then we have
  \[A(R,r,C_0(1+r))\leq (1+Cc)\overline{A}(R)+c^{-C}\]
  for any $0< c\leq 2^{-C_0}$.
  \end{corollary}
 
 \begin{proof}
 We can assume that $r < C_0^CR $ since the claim otherwise follows from Proposition \ref{p5.5}. Let $J$ be the least integer such that  $r\geq C_0^{-J}C_0^CR$. Since $r\geq R^{1/2+4/N}$, this implies $J \lesssim \log r$. Define $r:=r_0 >r_1>\ldots>r_J$ recursively by $r_{j+1}=r_j(1-Cr_j^{-1/3N})$. The sequence $\{r_j\}_{j=0}^J$ decreases  slowly, and has only about $\log r$ terms, thus $r_J \approx r$. For $0 < j  \leq J$ define $c_j:=C_0^{-1}cC_0^{(j-J)\epsilon/4C}$. Using these values of $c_j$ and $r_j$ we iterate Proposition \ref{p5.6} to obtain
 \[A(R,r,C_0(1+r)) \leq (1+Cc)A(R/C_0^J,r_J,C_0(1+r))+c^{-C}.\]
 Now we can use Proposition \ref{p5.5}, which applied to right hand side yields the desired result.
  
 \end{proof}

 Now we are in a position to show (\ref{eq2.2}). Combining the last result with Proposition \ref{p5.3} and setting $c=2^{-C_1}$ one sees that 
 \[A(R)\leq (1-C_0^{-C}) \overline{A}(R) +2^{CC_1}\]
 holds if $R \geq 2^{NC_1}$. Combining with (\ref{cn}) this extends to $R \geq 2C_02^{C_1}$. Using (\ref{pba}) we further extend it to $R \geq 2^{C_1/2}$, and hence can take the supremum for the left hand side to get
 \[\overline{A}(R)\leq (1-C_0^{-C}) \overline{A}(R) +2^{CC_1}\]
 for all $R\geq 2^{C_1/2}$. From this clearly (\ref{eq2.2}) follows.


 \end{document}